\documentclass[12pt,reqno]{amsart}

\usepackage{amsthm, mathrsfs}
\usepackage{enumerate}

\theoremstyle{plain}
\newtheorem{thm}{\sc Theorem}[section]
\newtheorem{defn}[thm]{\sc Definition}
\newtheorem{lem}[thm]{\sc Lemma}
\newtheorem{prop}[thm]{\sc Proposition}
\newtheorem{cor}[thm]{\sc Corollary}
\newtheorem{rem}[thm]{\sc Remark}

\title[Representations of $A_\Phi(G)$]{Representations of the Orlicz Fig\`{a}-Talamanca Herz algebras and Spectral Subspaces}
\author{Rattan Lal}
\address{Rattan Lal,\newline\indent Department of Mathematics,\newline\indent Indian Institute of Technology Delhi,\newline\indent Delhi - 110016, India.}
\email{rattanlaltank@gmail.com}
\author{N. Shravan Kumar}
\address{N. Shravan Kumar,\newline\indent Department of Mathematics,\newline\indent Indian Institute of Technology Delhi,\newline\indent Delhi - 110016, India.}
\email{shravankumar@maths.iitd.ac.in}

\begin{document}

\begin{abstract}
Let $G$ be a locally compact group. In this note, we characterise non-degenerate *-representations of $A_\Phi(G)$ and $B_\Phi(G).$ We also study spectral subspaces associated to a non-degenerate Banach space representation of $A_\Phi(G).$
\end{abstract}

\keywords{Orlicz Fig\`{a}-Talamanca Herz algebra, representation, spectral subspaces}

\subjclass[2010]{Primary 43A15, 46J25; Secondary 22D99}

\maketitle

\section{Introduction}
Let $G$ be a locally compact group. It is well known that there is a one to one correspondence between the unitary representations of $G$ and the non-degenerate *-representations of $L^1(G)$ \cite[Pg. 73]{F}. Similarly, if $X$ is any locally compact Hausdorff space, then there is a one to one correspondence between the cyclic *-representations of $C_0(X)$ and positive bounded Borel measures on $G$ \cite[Pg. 486]{HR}. The corresponding result for the Fourier algebra $A(G)$ of a locally compact group is due to Lau and Losert \cite{LL}. Recently, Guex \cite{MSG} extended the result of Lau and Losert to Fig\`{a}-Talamanca Herz algebras.

In \cite{RS}, the authors have introduced and studied the $L^\Phi$-versions of the Fig\`{a}-Talamanca Herz algebras. Here $L^\Phi$ denotes the Orlicz space corresponding to the Young function $\Phi.$ The space $A_\Phi(G)$ is defined as the space of all continuous functions $u,$ where $u$ is of the form $$u=\underset{n=1}{\overset{\infty}{\sum}}f_n*\check{g_n},$$ where $f_n\in L^\Phi(G),$ $g_n\in L^\Psi(G),$ $(\Phi,\Psi)$ is a pair of complementary Young functions satisfying the $\Delta_2$-condition and $$\underset{n=1}{\overset{\infty}{\sum}}N_\Phi(f_n)\|g_n\|_\psi<\infty.$$ It is shown in \cite{RS} that $A_\Phi(G)$ is a regular, tauberian, semisimple commutative Banach algebra with the Gelfand spectrum homeomorphic to $G.$ For further details about these algebras we refer to \cite{RS}. 

This paper has the modest aim of characterising the non-degenerate *-representations of $A_\Phi(G)$ in the spirit of \cite{LL}. This characterisation is given in Corollary \ref{CNDSRA}. In section 3, we show that any non-degenerate *-representation of $A_\Phi(G)$ can be extended uniquely to a non-degenerate *-representation of $B_\Phi(G).$

Godement in his fundamental paper \cite{G} on Wiener Tauberian theorems studied spectral subspaces associated to a certain Banach space representations. This result was extended to the Fourier algebra $A(G)$ by Parthasarathy and Prakash \cite{PP}. In Section 4, we also study spectral subspaces of $A_\Phi(G).$

We shall follow the notations as in \cite{RS}. For any undefined notations or definitions the reader is asked to refer the above paper. 

\section{Non-degenerate $\ast$-representations of $A_\Phi(G)$}
In this section, motivated by the results of \cite{LL,MSG}, we describe all the non-degenerate *-representations of $A_\Phi(G).$ Throughout this section and the next, $\mathcal{H}$ will denote a Hilbert space.
\begin{prop}\label{CV}
Let $\mu$ be a positive measure (not necessarily bounded). 
\begin{enumerate}[(i)]
\item For each $u\in A_\Phi(G),$ the mapping $\pi_{\mu}(u):f\mapsto uf$ is a bounded linear operator on $L^2(G,d\mu).$
\item The mapping $ u \mapsto \pi_\mu(u)$ defines a $\ast$-representation of $A_\Phi(G)$ on $\mathcal{B}(L^2(G,d\mu)).$
\item If $\mu$ is bounded, then $\pi_\mu$ is a cyclic representation of $A_\Phi(G)$ with the constant $1$ function as cyclic vector.
\item If $\mu$ is arbitrary, then $\pi_\mu$ is non-degenerate.
\end{enumerate}
\end{prop}
\begin{proof}
(i) and (ii) are just a routine check.

(iii) We show that the constant $1$ function is a cyclic vector. Since the measure $\mu$ is finite, the conclusion follows from the density of $A_\Phi(G)\cap C_c(G)$ in $C_c(G)$ with respect to the $L^2(G,d\mu)$-norm.

(iv) Let $\mu$ be an arbitrary measure on $G.$ By \cite[Pg. 33, 2.2.7]{Dix}, it is enough to show that the representation $\pi_{\mu}$ is a direct sum of cyclic representations. By \cite[INT IV.77]{BB} and \cite[INT V.14, Proposition 4]{BB}, it follows that $$L^2(G,d\mu)\cong\underset{\alpha\in\wedge}{\oplus}L^2(G,d\mu_\alpha),$$ where $\{\mu_\alpha\}_{\alpha\in\wedge}$ is a summable family of measures with pairwise disjoint support. Now the conclusion follows from (iii).
\end{proof}
In the next result, we characterise all cyclic *-representations.
\begin{thm}\label{UE}
Let $\{\pi,\mathcal{H}\}$ be a cyclic $\ast$-representation of $A_\Phi(G).$ Then there exists a bounded positive measure $\mu\in M(G)$ such that $\pi$ is unitarily equivalent to the representation $\{\pi_{\mu},L^2(G,d\mu)\}$ given in Proposition \ref{CV}.
\end{thm}
\begin{proof}
Let $u\in A_\Phi(G).$ Then, by \cite[Pg. 22]{T}, it follows that $\|\pi(u)\|_{sp}\leq \|u\|_{sp}.$ By \cite[Theorem 3.4]{RS}, $A_\Phi(G)$ is a commutative Banach algebra and hence the spectral norm and the operator norm for $\pi(u)$ coincides. Further, as $A_\Phi(G)$ is semi-simple, $\|u\|_{sp}=\|u\|_\infty.$ Thus, $$\|\pi(u)\|_{\mathcal{B}(\mathcal{H})}\leq\|u\|_\infty.$$ As a consequence of this inequality and the fact that $A_\Phi(G)$ is dense in $C_0(G)$, it follows that $\pi$ extends to a *-representation of $C_0(G)$ on $\mathcal{H},$ still denoted as $\pi.$ Note that $\pi$ is a cyclic *-representation of the C*-algebra $C_0(G).$ Let $\varphi$ be the cyclic vector of the representation $\{\pi,C_0(G)\}.$ Define $T_\varphi:C_0(G)\rightarrow\mathbb{C}$ as $$T_\varphi(u)=\langle \pi(u)\varphi,\varphi\rangle,\ u\in C_0(G).$$ It is clear that $T_\varphi$ is a positive linear functional on $C_0(G)$ and hence, by Riesz representation theorem, there exists a bounded positive measure $\mu\in M(G)$ such that 
\begin{equation}\label{RRT}
T_\varphi(u)=\int_Gu\ d\mu.
\end{equation} 
Let $\pi_{\mu}$ denote the cyclic *-representation of $A_\Phi(G)$ on $L^2(G,d\mu),$ given by Proposition \ref{CV}. 

We now claim that the representations $\pi$ and $\pi_{\mu}$ of $A_\Phi(G)$ are unitarily equivalent. Since $\varphi$ is a cyclic vector, in order to prove the above claim, it is enough to show that the correspondence $\pi(u)\varphi\mapsto u.1$ is an isometry and commutes with $\pi$ and $\pi_{\mu}.$ Note that the above correspondence is well-defined by (\ref{RRT}). Let $T$ denote the above well-defined correspondence. 

We now show that $T$ is an isometry. Let $u\in A_\Phi(G).$ Then 
\begin{align*}
\langle \pi(u)\varphi,\pi(u)\varphi\rangle =& \langle \pi^*(u)\pi(u)\varphi,\varphi\rangle\\ =& \langle \pi(\bar{u}u)\varphi,\varphi\rangle\ (\pi\mbox{ is a *-homomorphism})\\ =& \int_G |u|^2\ d\mu = \langle \varphi,\varphi\rangle.
\end{align*}
Finally, we show that $T$ intertwines with $\pi$ and $\pi_\mu.$ Let $u\in A_\Phi(G).$ Then, for $v\in A_\phi(G),$ we have,
\begin{align*}
T(\pi(u)(\pi(v)\varphi)) =& T((\pi(u)\pi(v))\varphi) \\ =& T(\pi(uv)\varphi) = uv.1 \\ =& \pi_\mu(u)(v.1)=\pi_{\mu}(u)(T(\pi(v)\varphi)). \qedhere
\end{align*}
\end{proof}
Here is the main result of this section, describing all the non-degenerate Hilbert space representations of $A_\Phi(G).$
\begin{cor}\label{CNDSRA}
If $\{\pi,\mathcal{H}\}$ is any non-degenerate *-representation of $A_\Phi(G)$ then $\pi$ is unitarily equivalent to $\{\pi_{\mu},L^2(G,d\mu)\}$ for some positive measure $\mu.$
\end{cor}
\begin{proof}
Let $\{\pi,\mathcal{H}\}$ be a non-degenerate *-representation of $A_\Phi(G).$ By \cite[Proposition 2.2.7]{Dix}, $\pi$ is a direct sum of cyclic *-representations $\{\pi_\alpha,\mathcal{H}_\alpha\}_{\alpha\in\wedge}.$ For each $\alpha\in\wedge,$ by Theorem \ref{UE}, there exists a bounded positive measure $\mu_\alpha$ such that the representations $\{\pi_\alpha,\mathcal{H}_\alpha\}$ and $\{\pi_{\mu_\alpha},L^2(G,d\mu_\alpha)\}$ are unitarily equivalent. 

Suppose that the family $\{\mu_\alpha\}_{\alpha\in\wedge}$ is summable. Let $\mu=\underset{\alpha\in\wedge}{\sum}\mu_\alpha.$ Then $\mu$ will be a positive measure and $$\{\pi_{\mu}, L^2(G,d\mu)\}\cong\underset{\alpha\in\wedge}{\oplus}\{\pi_{\mu_\alpha}, L^2(G,d\mu_\alpha)\}\cong\underset{\alpha\in\wedge}{\oplus}\{\pi_\alpha,\mathcal{H}_\alpha\}\cong\{\pi,\mathcal{H}\}.$$

Thus, we are done if we can show that $\{\mu_\alpha\}_{\alpha\in\wedge}$ is a summable family. Let $f:G\rightarrow\mathbb{C}$ be a continuous function with compact support. Then, 
\begin{align*}
\underset{\alpha\in\wedge}{\sum}|\mu_\alpha(f)| =& \underset{\alpha\in\wedge}{\sum}\left|\int_Gf\ d\mu_\alpha\right| \leq \underset{\alpha\in\wedge}{\sum}\int_G\left|f\right|\ d\mu_\alpha \\ \leq& \underset{\alpha\in\wedge}{\sum}\left(\int_G\left|f\right|^2\ d\mu_\alpha\right)^{1/2} \left(\int_G\left|1\right|^2\ d\mu_\alpha\right)^{1/2} \\ =& \underset{\alpha\in\wedge}{\sum}\left(\int_G\left|f\right|^2\ d\mu_\alpha\right)^{1/2} \left(\mu_\alpha(G)\right)^{1/2} \\ \leq& \underset{\alpha\in\wedge}{\sup}\left(\mu_\alpha(G)\right)^{1/2}\underset{\alpha\in\wedge}{\sum}\left(\int_G\left|f\right|^2\ d\mu_\alpha\right)^{1/2} \\ \leq& \left(\underset{\alpha\in\wedge}{\sup}\ \mu_\alpha(G)\right)^{1/2}\underset{\alpha\in\wedge}{\sum}\left(\int_G\left|f\right|^2\ d\mu_\alpha\right)^{1/2} <\infty,
\end{align*} which follows from the boundedness of $\mu_\alpha$'s and the fact that $\underset{\alpha\in\wedge}{\sum}\left(\int_G\left|f\right|^2\ d\mu_\alpha\right)^{1/2}$ is finite.
\end{proof}

\section{Non-degenerate $\ast$-representations of $B_\Phi(G)$}
In this section, we show that the non-degenerate representations described in the previous section can be extended uniquely to $B_\Phi(G).$
\begin{thm}\label{RMA}
Let $\{\pi,\mathcal{H}\}$ be a non-degenerate *-representation of $A_\Phi(G).$ 
\begin{enumerate}[(i)]
\item For each $u\in B_\Phi(G),$ there exists a unique operator $\widetilde{\pi}(u)\in\mathcal{B}(\mathcal{H})$ such that, $\forall\ v\in A_\Phi(G),$ \begin{equation}\label{EU}
\widetilde{\pi}(u)\pi(v)=\pi(uv)
\end{equation}and
\begin{equation}\label{rest}
\widetilde{\pi}(v)=\pi(v).
\end{equation}
\item The mapping $u\mapsto\widetilde{\pi}(u)$ defines a non-degenerate *-representation of $B_\Phi(G)$ on $\mathcal{H}.$
\end{enumerate}
\end{thm}
\begin{proof}
{\bf (i)} Let $\pi$ be a non-degenerate *-representation of $A_\Phi(G).$ By \cite[Proposition 2.2.7]{Dix}, $\pi$ is a direct sum of cyclic *-representations, say $\{\pi_\alpha,\mathcal{H}_\alpha\}_{\alpha\in\wedge}.$ If we can prove (i) for each of these $\pi_\alpha$'s, then the argument for $\pi$ is similar to the one given in Corollary \ref{CNDSRA}. Thus, in order to prove this, we assume that the representation $\pi$ is cyclic. Since $\pi$ is a cyclic *-representation, by Theorem \ref{UE}, $\pi$ is unitarily equivalent to $\pi_\mu,$ for some bounded positive measure $\mu.$ So, without loss of generality, let us assume that the non-degenerate *-representation of $A_\Phi(G)$ is $\pi_{\mu}$ for some bounded positive measure $\mu.$

Let $u\in B_\Phi(G).$ By Proposition \ref{CV}, the space $\mathcal{K}:=span\{\pi_\mu(v).1:v\in A_\Phi(G)\}$ is dense in $L^2(G,d\mu).$ Define $\widetilde{\pi_\mu}(u):\mathcal{K}\rightarrow L^2(G,d\mu)$ as $$\widetilde{\pi_\mu}(u)(\pi_\mu(v).1)=\pi_\mu(uv).1.$$ It is clear that $\widetilde{\pi_\mu}(u)$ is linear. We now claim that $\widetilde{\pi_\mu}(u)$ is bounded. Let $v\in A_\Phi(G).$ Then \begin{align*}
\|\widetilde{\pi_\mu}(u)\left(\pi_{\mu}(v).1\right)\|_2^2 =& \|\pi_{\mu}(uv).1\|_2^2 \\ =& \int_G\left|\pi_{\mu}(uv).1\right|^2\ d\mu \\ =& \int_G\left|uv\right|^2\ d\mu \\ \leq& \|u\|_\infty^2 \int_G |v|^2\ d\mu \leq \|u\|_{B_\Phi}^2\|\pi_{\mu}(v).1\|_2^2.
\end{align*}
Thus, $\widetilde{\pi_\mu}(u)$ extends to a bounded linear operator on $L^2(G,d\mu),$ still denoted $\widetilde{\pi_\mu}(u).$ Further, it is clear that, for $u\in B_\Phi(G)$ and $v\in A_\Phi(G),$ $\widetilde{\pi_\mu}(u)\pi_{\mu}(v)=\pi_{\mu}(uv).$ Now, let $v\in A_\Phi(G).$ Then, for $u\in A_\Phi(G),$ $$\widetilde{\pi_\mu}(v)(\pi_{\mu}(u).1)=\pi_{\mu}(vu).1=\pi_{\mu}(v)\left(\pi_{\mu}(u).1\right).$$ Again, as $\mathcal{K}$ is dense in $L^2(G,d\mu),$ it follows that $\widetilde{\pi_\mu}(v)=\pi_\mu(v)$ for all $v\in A_\Phi(G).$

Finally, uniqueness follows from condition (\ref{EU}).

{\bf (ii)} Non-degeneracy of $\widetilde{\pi}$ follows from the fact that $\pi$ is non-degenerate. Further, homomorphism property of $\widetilde{\pi}$ follows from (\ref{EU}). Now, we show that $\widetilde{\pi}$ preserves involution. Let $u\in B_\Phi(G).$ Then, for $v\in A_\Phi(G)$ and $\xi,\eta\in \mathcal{H},$ we have
\begin{eqnarray*}
\langle\widetilde{\pi}(u)^*\pi(v)\xi,\eta\rangle &=& \langle\xi,\pi(\overline{v})\widetilde{\pi}(u)\eta\rangle \\ &=& \langle\xi,\widetilde{\pi}(\overline{v})\widetilde{\pi}(u)\eta\rangle\ \mbox{(by (\ref{rest}))} \\ &=& \langle\xi,\widetilde{\pi}(u\overline{v})\eta\rangle\ \mbox{(}\widetilde{\pi}\mbox{ is a homomorphism)}\\ &=& \langle\xi,\pi(u\overline{v})\eta\rangle\ \mbox{(by (\ref{rest}))} \\ &=& \langle\xi,\pi(\overline{u\overline{v}})^*\eta\rangle\ \mbox{(}\pi\mbox{ preserves involution)} \\  &=& \langle\pi(\overline{u}v)\xi,\eta\rangle \\ &=& \langle\widetilde{\pi}(\overline{u})\pi(v)\xi,\eta\rangle.\ \mbox{(by (\ref{EU}))}
\end{eqnarray*}
Since the representation $\pi$ is non-degenerate, the space $\{\pi(u)\xi:u\in A_\Phi(G),\xi\in\mathcal{H}\}$ is dense in $\mathcal{H}.$ Thus, it follows that $\widetilde{\pi}(u)^*=\widetilde{\pi}(\overline{u})$ for all $u\in B_\Phi(G).$
\end{proof}
The following corollary is the converse of the above theorem.
\begin{cor}
Let $\{\pi,\mathcal{H}\}$ be a *-representation of $B_\Phi(G)$ such that $\pi|_{A_\Phi}$ is non-degenerate. Then, $\widetilde{\pi|_{A_\Phi}}=\pi$ and $\pi$ is non-degenerate.
\end{cor}
\begin{proof}
Let $u\in B_\Phi(G)$ and $v\in A_\Phi(G).$ Then $$\pi(u)\pi|_{A_\Phi}(v)=\pi(u)\pi(v)=\pi(uv)=\pi|_{A_\Phi}(uv).$$ Thus, by Theorem \ref{RMA}, it follows that $\widetilde{\pi|_{A_\Phi}}=\pi.$ Again by Theorem \ref{RMA}, $\widetilde{\pi|_{A_\Phi}}$ is non-degenerate and hence it follows that the representation $\pi$ is non-degenerate.
\end{proof}

\section{Spectral subspaces}
In this section, we study the spectral subspaces associated to a non-degenerate Banach space representation of $A_\Phi(G).$ Our main aim in this section is to prove Corollary \ref{TSC}. Most of the ideas of this section are taken from \cite{PP}.
\begin{defn}
Let $T\in PM_\Psi(G).$ Then the support of $T$ is defined as $$supp(T)=\{x\in G:u\in A_\Phi(G),u(x)\neq0\Rightarrow u.T\neq0\}.$$
\end{defn} 
Here we recall some of the properties of the support of $T$ in the form of a Lemma \cite[Pg. 101]{H}.
\begin{lem}\label{supp_prop}\mbox{ }
\begin{enumerate}[(i)]
\item If $T_1,T_2\in PM_\Psi(G)$ then $supp(T_1+T_2)\subseteq supp(T_1)\cup supp(T_2).$
\item If $u\in A_\Phi(G)$ and $T\in PM_\Psi(G)$ then $supp(u.T)\subseteq supp(u)\cap supp(T).$
\item If $c\in\mathbb{C}$ and $T\in PM_\Psi(G)$ then $supp(cT)\subseteq supp(T).$
\item Let $T\in PM_\Psi(G)$ and let $E$ be a closed subset of $G.$ If a net $\{T_\alpha\}\subset PM_\Psi(G)$ converges weakly to $T$ with $supp(T_\alpha)\subset E$ for all $\alpha,$ then $supp(T)\subset E.$
\end{enumerate}
\end{lem}
Let $X$ be a Banach space and let $\pi$ be an algebra representation of $A_\Phi(G)$ on $X.$ For $\varphi\in X$ and $x^*\in X^*,$ define $T_{x^*,\varphi}:A_\Phi(G)\rightarrow\mathbb{C}$ as $$\langle u,T_{x^*,\varphi}\rangle:=\langle \pi(u)\varphi,x^*\rangle\ \forall\ u\in A_\Phi(G).$$ We say that the representation $\pi$ is continuous if $T_{x^*,\varphi}$ is a continuous linear functional on $A_\Phi(G)$ for each $\varphi\in X$ and $x^*\in X^*.$ It follows from uniform boundedness principle that the linear map $\pi:A_\Phi(G)\rightarrow\mathcal{B}(X)$ is norm continuous.

From now onwards, $X$ will denote a Banach space and $\pi$ an algebra representation of $A_\Phi(G)$ on $X.$

Let $E$ be a closed subset of $G.$ Define $$X_E:=\{\varphi\in X:supp(T_{x^*,\varphi})\subseteq E\ \forall\ x^*\in X^*\}.$$
\begin{rem}\label{WS}
An immediate consequence of the above definition is that, if $E=G$ then $X_E=X.$
\end{rem}
\begin{lem}
The set $X_E$ is a closed $\pi$-invariant subspace of $X.$
\end{lem}
\begin{proof}
Note that for any $x^*\in X^*,$ $\varphi_1,\varphi_2\in X_E$ and $\alpha\in\mathbb{C},$ we have $$T_{x^*,\varphi_1+\alpha \varphi_2}=T_{x^*,\varphi_1}+\alpha T_{x^*,\varphi_2}.$$ Thus, it follows from (i) and (iii) of Lemma \ref{supp_prop} that $X_E$ is a linear space. Further, closedness of $X_E$ is an immediate consequence of (iv) from Lemma \ref{supp_prop}. Again, note that, for any $u\in A_\Phi(G),\varphi\in X$ and $x^*\in X^*,$ we have $T_{x^*,\pi(u)\varphi}=u.T_{x^*,\varphi}$ and hence the invariance of $X_E$ under $\pi$ follows from (ii) of Lemma \ref{supp_prop}.
\end{proof}
The subspace $X_E$ is called as the spectral subspace associated with the representation $\pi$ and the closed set $E.$
\begin{lem}\label{SSProp}
Let $\pi$ be a non-degenerate representation of $A_\Phi(G).$
\begin{enumerate}[(i)]
\item The space $X_\emptyset=\{0\}.$
\item If $\{E_i\}$ is an arbitrary collection of closed subsets of $G,$ then $X_{\underset{i}{\cap}E_i}=\underset{i}{\cap}X_{E_i}.$
\end{enumerate}
\end{lem}
\begin{proof}
(i) is an easy consequence of the non-degeneracy of $\pi,$ while (ii) is trivial.
\end{proof}
The following is an immediate corollary of Remark \ref{WS} and Lemma \ref{SSProp}.
\begin{cor}\label{SSS}
There exists a smallest closed non-empty set $E$ of $G$ such that $X_E=X.$
\end{cor}
\begin{prop}
Let $K_1$ and $K_2$ be disjoint compact subsets of $G.$ Then $X_{K_1\cup K_2}=X_{K_1}\oplus X_{K_2}.$
\end{prop}
\begin{proof}
The proof of this follows exactly as given in \cite[Proposition 2 (iii)]{PP}.
\end{proof}
\begin{thm}\label{TSSS}
Let $\pi$ be a non-degenerate representation of $A_\Phi(G)$ such that the only spectral subspaces are the trivial subspaces. Then there exists $x\in G$ such that $X_{\{x\}}=X.$
\end{thm}
\begin{proof}
Choose a smallest non-empty closed set $E$ such that $X_E=X,$ which is possible by Corollary \ref{SSS}. Suppose there exists $x,y\in E$ such that $x\neq y.$ As $G$ is locally compact and Hausdorff, there exists an open set $U$ and a compact set $K$ such that $x\in U\subset K$ and $y\notin K.$ Since $A_\Phi(G)$ is regular, there exists $u\in A_\Phi(G)$ such that $u=1$ on $U$ and $supp(u)\subset K.$

Let $v\in A_\Phi(G)$ be arbitrary. Let $v_1=v-uv$ and $v_2=uv$ so that $v=v_1+v_2.$ Let $V=\{z\in G:v_1(z)\neq 0\}.$ The choice of $u$ tells us that $x\notin\overline{V}.$ Again, using the regularity of $A_\Phi(G),$ choose a function $w\in A_\Phi(G)$ such that $w=1$ on some open set $W$ containing $x$ and $supp(w)\cap V=\emptyset.$ Further, it is clear that $v_1w=0.$

We now claim that $\pi(v)=0.$ Let $\varphi\in X$ and $x^*\in X^*.$ If $z\in W,$ then $w(z)=1$ and hence $T_{x^*,\pi(v_1w)\varphi}=0$ as $T_{x^*,\pi(v_1w)\varphi}=w.T_{x^*,\pi(v_1)\varphi}.$ Thus $supp(T_{x^*,\pi(v_1)\varphi})\subset W^c.$ Therefore, using the non-degeneracy of $\pi,$ it follows that, if $\pi(v_1)\varphi\neq 0$ then $X_{W^c}=X$ and hence, by the choice of the set $E,$ it follows that $E$ is a subset of $W^c.$ On the other hand, $x\notin W^c$ and $x\in E$ and hence $E$ is not a subset of $W^c.$ Therefore, $\pi(v_1)=0.$ Similarly, one can show that $\pi(v_2)=0.$ Thus $\pi(v)=0.$ Since $v$ is arbitrary, it follows that $\pi(v)=0$ for all $v\in A_\Phi(G),$ which is a contradiction. Thus the set $E$ is a singleton.
\end{proof}
\begin{cor}\label{TSC}
Let $\pi$ be a non-degenerate representation of $A_\Phi(G)$ such that the only spectral subspaces are the trivial subspaces. Then $\pi$ is a character.
\end{cor}
\begin{proof}
By Theorem \ref{TSSS}, there exists $x\in G$ such that $X_{\{x\}}=X,$ i.e., $supp(T_{x^*,\varphi})\subset\{x\}$ for all $\varphi\in X$ and $x^*\in X^*.$ As singletons are sets of spectral synthesis for $A_\Phi(G)$ \cite[Theorem 3.6 (i)]{RS}, it follows that \begin{equation}\label{SS}
T_{x^*,\varphi}=c\delta_x
\end{equation} for some $c\in\mathbb{C}.$ Let $u\in A_\Phi(G)$ such that $u(x)=1.$ Then \begin{equation}\label{c}
c=c\langle u,\delta_x\rangle=\langle u,c\delta_x\rangle=\langle u,T_{x^*,\varphi}\rangle=\langle \pi(u)\varphi,x^*\rangle.
\end{equation}

We now claim that $\pi$ is a character. Let $v\in A_\Phi(G).$ Then, for $\varphi\in X$ and $x^*\in X^*,$ we have 
\begin{align*}
\langle \pi(v)\varphi,x^* \rangle =& \langle v, T_{x^*,\varphi} \rangle = \langle v, c\delta_x \rangle\ (\mbox{by }(\ref{SS})) \\ =& c\langle v, \delta_x \rangle = \langle \pi(u)\varphi,x^*\rangle\langle v, \delta_x \rangle\ (\mbox{by }(\ref{c}))\\ =& v(x)\langle \pi(u)\varphi,x^*\rangle = \langle v(x)\pi(u)\varphi,x^*\rangle.
\end{align*}
Since $\varphi$ and $x^*$ are arbitrary, it follows that $\pi(v)=u(x)\pi(u).$ Now $$\pi(u)=u(x)\pi(u)=u^2(x)\pi(u)=\pi(u^2)=\pi(u)^2,$$ i.e., $\pi(u)$ is a projection.  As $\pi$ is non-degenerate, it follows that $\pi(u)$ is the identity operator $I$ on $X.$ Thus $$\pi(v)=v(x)I\ \forall\ v\in A_\Phi(G),$$ i.e., $\pi$ is a character. 
\end{proof}

\section*{Acknowledgement}
The first author would like to thank the University Grants Commission, India, for research grant.

\end{document}